\documentclass[a4paper,12pt]{article}

\usepackage{times, url}
\usepackage{amsmath,amssymb,amsthm,amsfonts,bm,float,cite}
\usepackage{booktabs}
\usepackage[table]{xcolor}
\usepackage{tabularx}
\usepackage{boldline}
\usepackage[utf8]{inputenc}

\newtheorem{definition}{Definition}[section]
\newtheorem{remark}{Remark}[section]
\newtheorem{theorem}[definition]{Theorem}
\newtheorem{lemma}[definition]{Lemma}
\newtheorem{corollary}[definition]{Corollary}

\usepackage[none]{hyphenat}
\usepackage[center]{caption}

\usepackage{mathtools}

\DeclarePairedDelimiter\floor{\lfloor}{\rfloor}

\usepackage[pdftex]{graphicx}
\usepackage{mathtools}
\DeclareGraphicsExtensions{.png}

\makeatletter
\renewcommand{\@biblabel}[1]{[#1]\hfill}
\makeatother

\begin{document}
\setcounter{page}{1}

\begin{center}
{\LARGE \bf Shortest polygonal chains covering each planar square grid\\[4mm] } 
\vspace{10mm}

{\Large \bf Marco Rip\`a}
\vspace{3mm}

World Intelligence Network\\ 
Rome, Italy\\
\end{center}
\vspace{12mm}


\noindent {\bf Abstract:} Given any $n \in \mathbb{Z}^{+}$, we constructively prove the existence of covering paths and circuits in the plane which are characterized by the same link length of the minimum-link covering trails for the two-dimensional grid $G_n^2 := \{0,1, \ldots, n-1\} \times \{0, 1, \ldots, n-1\}$. Furthermore, we introduce a general algorithm that returns a covering cycle of analogous link length for any even value of $n$. Finally, we provide the tight upper bound $n^2 - 3 + 5 \cdot \sqrt{2}$ units for the minimum total distance travelled to visit all the nodes of $G_n^2$ with a minimum-link trail (i.e., a trail with $2 \cdot n - 2$ edges if $n$ is above two).


\vspace{8mm}


\section{Introduction} \label{sec:Intr}

The classical \textit{nine dots puzzle} by Sam Loyd (see Reference \cite{Loyd:6}, p. 301) is a famous thinking outside-the-box challenge that asks to connect nine points arranged in a regular square grid with a polygonal chain having no more than four edges. This puzzle did not introduce the constraint to visit each point only once, nor to come back to the starting vertex with the last edge of the path.

The planar generalization of the mentioned problem to larger square grids, asking to join all the points of the set $\{\{0,1,\ldots, n-1\} \times \{0,1,\ldots, n-1\}\} \subset \mathbb{R}^2$ with a polygonal chain consisting of only $2 \cdot n-2$ segments \cite{Levitin:5, Pegg:7}, became quite popular on the Web at the beginning of the XXI century as a challenging brain teaser (for applications in cognitive psychology, see also \cite{Chein:2, Kershaw:3}).

In the next section, assuming $n > 4$, we constructively prove that all the nodes of the described $n \times n$ lattice can be joined by a standard path consisting of only $2 \cdot n - 2$ line segments and, furthermore, this pattern originates a valid closed path (i.e., a covering cycle for the given $n \times n$ lattice of points) of $2 \cdot n - 2$ segments for each even value of $n$. In order to compactly state these results, let us give a few definitions first, together with a couple of remarks.

\begin{definition} \label{def1.2}
Let $n \in \mathbb{Z}^{+}$. We define $G_n^2 := \{0,1,\ldots, n-1\} \times \{0,1,\ldots, n-1\}$ so that the grid $G_n^2$ is a set of $n^2$ points in the Euclidean space $\mathbb{R}^{2}$.
\end{definition}

\begin{definition} \label{def1.3}
Let $P_n$ indicate a covering path for $G_n^2$, which is a directed polygonal chain that visits each node of $G_n^2$ exactly once, and let $T_n$ denote a covering trail for $G_n^2$, which is a directed polygonal chain that joins all the nodes of the mentioned planar grid by visiting all of them at least once (e.g., the polygonal chain $(0,-1)$-$(0,3)$-$(3,0)$-$(0,0)$-$(2,2)$ is a covering trail for $G_3^2$, but it is not also a covering path since the vertex $(0,0)$ is visited twice); in both cases, edges cannot be repeated and each couple of consecutive edges cannot be collinear, while a covering path can repeat Steiner points since (by definition) they do not belong to the given grid (e.g., $(0,3)$-$(0,0)$-$(3,0)$-$(0,3)$-$(3,3)$-$(\frac{1}{7},\frac{1}{7})$ is a covering path for $G_3^2$ with Steiner points located at $(0,3)$, $(3,0)$, $(3,3)$, and $(\frac{1}{7},\frac{1}{7})$). Furthermore, we indicate as $\{T_n \}$ and $\{P_n \}$ the sets of all the covering trails and covering paths for $G_n^2$, respectively.
\end{definition}

\begin{definition} \label{def1.4}
Let $F_n$ indicate a covering circuit for $G_n^2$, which is a (non-empty) closed directed trail that visits each node of $G_n^2$ and whose starting point is equal to the endpoint. Then, let $C_n$ indicate a covering cycle for $G_n^2$, which is a (non-empty and possibly self-crossing) closed path that is also a covering circuit for $G_n^2$. Consequently, $C_n$ is such that only the starting node of the grid can be touched twice, as long as it coincides with the beginning and ending vertex of the path itself. Furthermore, we indicate as $\{F_n \}$ and $\{C_n \}$ the sets of all the covering circuits and covering cycles for $G_n^2$, respectively.
\end{definition}

\begin{definition} \label{def1.5}
The link length, $h(T_n)$, of a covering trail, $T_n$, corresponds to the number of its edges, while the length of each edge corresponds to the Euclidean distance between its two endpoints. Lastly, let us define the total distance travelled, $l(T_n )$, as the sum of the lengths of each edge belonging to the considered covering trail.
\end{definition}

\begin{remark} \label{rem1.1}
Since for the given $G_n^2$, by Definitions \ref{def1.3} and \ref{def1.4}, $T_n$ indicates a covering trail, $P_n$ a covering path, $F_n$ a covering circuit, and $C_n$ a covering cycle, their minimum link length must always satisfy the inequalities $\min(h\{T_n \}) \leq \min(h\{P_n \})$, $\min(h\{F_n \}) \leq \min(h\{C_n \})$, $\min(h\{T_n \}) \leq \min(h\{F_n \})$, and $\min(h\{P_n \}) \leq \min(h\{C_n \})$. In particular, $\min(h\{T_n \}) = 2 \cdot (n-1)$ is true if and only if $n$ is (strictly) greater than two.
\end{remark}

\begin{remark} \label{rem1.2}
Given any $n \in \mathbb{Z}^{+}$, the sets $\{T_n \}$, $\{P_n \}$, $\{F_n \}$, and $\{C_n \}$ cannot be characterized by a cardinality smaller than $2^{\aleph_0 }$, and so are the subsets consisting of all the covering trails/paths/ \linebreak circuits/cycles of minimum link length. Thus, to avoid ambiguity, we will describe a specific element of one of the above-mentioned subsets providing the sequence of its vertices (e.g., \linebreak $(0,0)$-$(0,2)$-$(2,0)$-$(0,0)$ and $\left(-\frac{1}{\sqrt{3}},0 \right)$-$\left(1+\frac{1}{\sqrt{3}},0 \right)$-$\left(\frac{1}{2},\frac{\sqrt{3}}{2}+1 \right)$-$\left(-\frac{1}{\sqrt{3}},0 \right)$ are both elements of the unique subset of $\{C_2 \}$ which contains all and only the minimal covering cycles for $G_n^2$, and it will be clear by context which covering cycle we are referring to as $``C_2"$).
\end{remark}

\begin{definition} \label{def1.6}
Let $l_l \{T_n \}$ and $l_u \{T_n \}$ denote the (current) lower and upper bound (respectively) for the minimum total distance travelled to cover $G_n^2$ with a minimum-link trail so that $l_l \{T_n \} \leq \min(l\{T_n : h\{T_n \} = \min(h\{T_n \})\}) \leq l_u \{T_n \}$.
\end{definition}

Although it could seem noteworthy to take Loyd's puzzle in more than two dimensions by extending $G_3^2$ to $G_3^k$ for any given $k \in \mathbb{N}-\{0,1,2\}$ \cite{Bereg:1}, the classical problem has been constructively solved in recent years by providing an algorithm which returns optimal covering trails (inside the axis-aligned bounding box $[0,3] \times [0,3] \times \dots \times [0,3] \subset \mathbb{R}^k$) with $\frac{3^k-1}{2}$ edges \cite{Ripa:8}, while the planar generalization has remained (partially) unsolved under a few very basic constraints, such as avoiding to visit any node of $G_n^2$ more than once or to come back to the starting point with the last edge. And so it was that, in September 2021, Aldo Tragni shared with us a clever pattern, which returns a covering cycle of link length $2 \cdot (n-1)$ for any even value of $n$ above two. This finally proves Michael Cysouw's first conjecture (as stated in \cite{Pegg:7}) concerning $(2 \cdot n) \times (2 \cdot n)$ grids and also extends Joseph DeVincentis' result \cite{Pegg:7} from $G_n^2 : n 	\equiv 2 \pmod {4}$ to any $G_n^2 : n \equiv 0 \pmod {2}$.

Now, since these cycles can always be converted into open paths by cutting them near one of their $2 \cdot n-2$ turning points, and considering that Keszegh \cite{Keszegh:4} has shown that the same lower bound applies to any minimum-link covering path for $G_n^2$, the main related open problem is to prove the existence of optimal covering paths also for every odd value of $n$ (confirming Kato's conjecture, as reported by Edward Taylor Pegg Jr. in Reference \cite{Pegg:7}).

Thus, Section \ref{sec:2} is devoted to showing how to constructively solve the $n$ odd case and it also proves the existence of covering circuits of link length $2 \cdot n-2$ for any $n$ above three.

In Section \ref{sec:3}, we consider the standard Euclidean metrics and provide tight bounds for the minimum total distance travelled to cover any $G_n^2$ with a minimum-link trail.


\section{Minimal covering paths and cycles} \label{sec:2}

In early September 2021, Rip\`a and Baglioni proved the following lemma (see Appendix for its proof).

\begin{lemma} \label{Lemma 2.1}
$\forall n \in \mathbb{Z}^+$, $\exists P_n : h(P_n ) = \min(h\{T_n \})$ (i.e., $h(P_1 ) = 1$, $h(P_2 ) = 3$, and $n : n \geq 3 \Rightarrow h(P_n ) = 2 \cdot (n-1)$).
\end{lemma}

Thanks to the contribution of Aldo Tragni, we soon realized that this preliminary result could have been easily incorporated into the statement of a theorem, which enhance it (by including optimal covering cycles for every $G_n^2 : n \equiv 0 \pmod{2}$) and also greatly simplify its proof.

\begin{theorem} \label{Theorem 2.1}
Let $n, m \in \mathbb{Z}^+$. $\forall n$, $\exists P_n : h(P_n ) = \min(h\{T_n \})$. In particular, if $n = 2 \cdot m$, then $\exists C_n : h(C_n ) = \min(h\{T_n \})$.
\end{theorem}

\begin{proof}
We constructively prove Theorem \ref{Theorem 2.1} by simply observing that, given $G_n^2 : n \in \mathbb{N}-\{0,1,2,3,4\}$, the implicit algorithm which returns the minimal covering paths shown in Figure \ref{fig:CP_Figure_1} also produces covering cycles of link length $2 \cdot (n-1)$ for any even value of $n$ above four. In details, if $n$ is greater than four and is also even, by the universal pattern described through Figure \ref{fig:CP_Figure_2}, it is automatically possible to close all the paths, proving the existence of minimal covering cycles with $\min(h\{T_n \})$ edges, for the given $\{0,1, \dots ,2 \cdot m -1\} \times \{0,1, \dots ,2 \cdot m -1\}$ grid, $m \in \mathbb{Z}^+$.

\begin{figure}[H]
\begin{center}
\includegraphics[width=\linewidth]{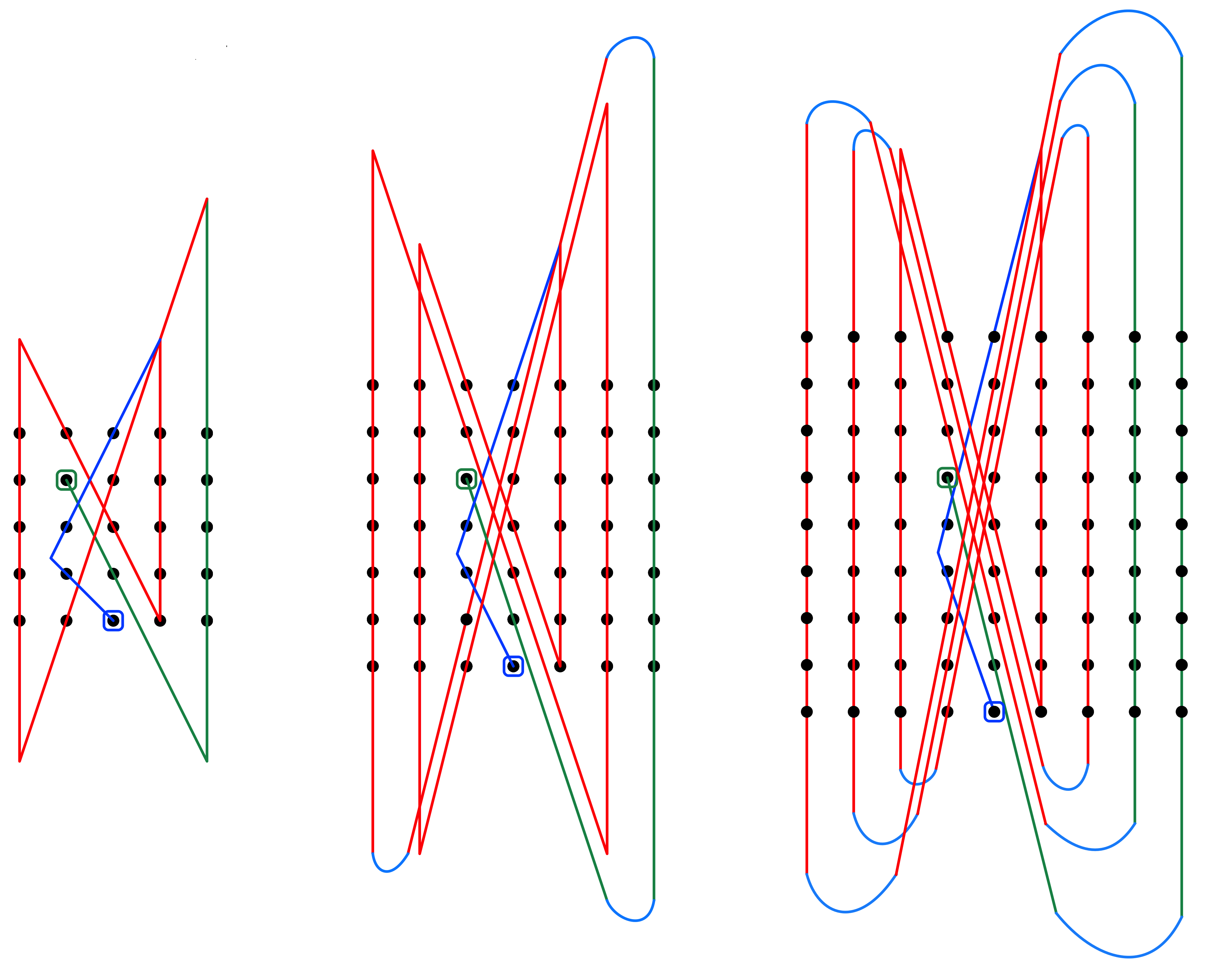}
\end{center}
\caption{Minimum-link covering paths for $G_5^2$, $G_7^2$, and $G_9^2$, constructed from the universal pattern provided. They prove that $\min(h\{T_n \}) = \min(h\{P_n \}) = 2 \cdot (n-1)$ holds for any odd integer $n$ which is above four.}
\label{fig:CP_Figure_1}
\end{figure}

\begin{figure}[H]
\begin{center}
\includegraphics[width=\linewidth]{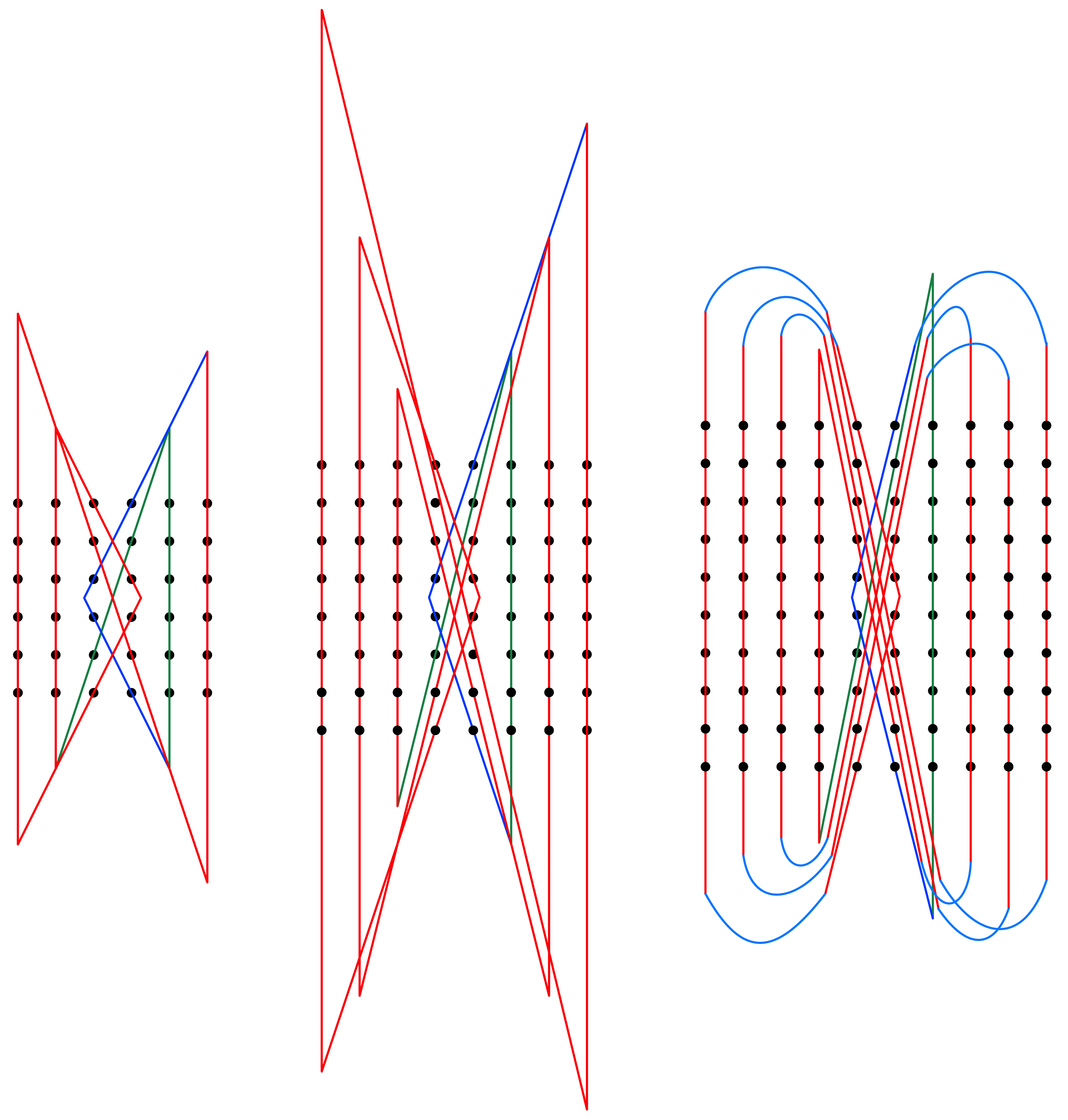}
\end{center}
\caption{Optimal covering cycles for $G_6^2$, $G_8^2$, and $G_{10}^2$, derived from the general algorithm introduced by Figure \ref{fig:CP_Figure_1}. They show that $\min(h\{T_n \}) = \min(h\{C_n \})$ holds for any even number $n$ which is above four.}
\label{fig:CP_Figure_2}
\end{figure}

Thus, if we assume that $m$ is above two, then our universal pattern produces a covering path for $G_{(2 \cdot m-1)}^2$ whose starting point $\left( \frac{n-1}{2},0 \right)$ and endpoint $\left( \frac{n-3}{2},\frac{n+1}{2} \right)$ do not coincide (Figure \ref{fig:CP_Figure_1}), whereas we have a covering cycle for any $G_{(2 \cdot m)}^2$ (see Figure \ref{fig:CP_Figure_2}).

To end the proof, we only need to check the four remaining cases, and they are shown in Figure \ref{fig:CP_Figure_3}.

\begin{figure}[H]
\begin{center}
\includegraphics[width=\linewidth]{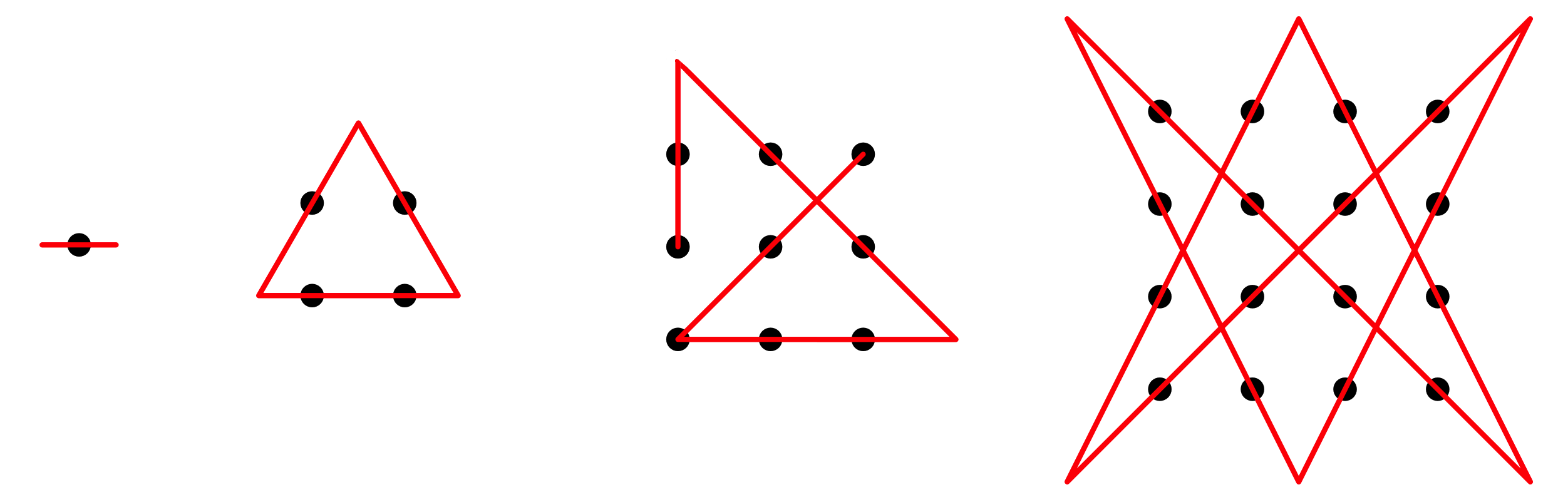}
\end{center}
\caption{Four optimal solutions for the $n \times n$ dots problem, $n=1,2,3,4$ (see Reference \cite{Pegg:7} for the charming covering cycle $C_4 := (-1,-1)$-$\left(\frac{3}{2},4 \right)$-$(4,-1)$-$(-1,4)$-$\left(\frac{3}{2},-1 \right)$-$(4,4)$-$(-1,-1)$).}
\label{fig:CP_Figure_3}
\end{figure}
This concludes the proof of Theorem \ref{Theorem 2.1}.	
\end{proof}

\begin{theorem} \label{Theorem 2.2}
Let $n \in \mathbb{N}-\{0,1,3\}$. $\forall G_{n}^2$, $\exists F_n : h(F_n ) = \min(h\{T_n \})$.
\end{theorem}

\begin{proof}
Although, from Theorem \ref{Theorem 2.1}, we already know that if $n : (n \equiv 0 \pmod{2} \wedge n > 2)$, then $G_{n}^2$ admits a covering cycle of link length $2 \cdot (n-1)$, here we provide an easy proof, which holds for every $n \in \mathbb{N}-\{0,1,3\}$, arising from the well-known square spiral pattern \cite{Keszegh:4}.

Trivially, if $n = 2$, then the covering cycle $C_2 := (0,0)$-$(0,2)$-$(2,0)$-$(0,0)$ is both a covering circuit and a covering path (by definition).

Thus, let us assume that $G_n : n \geq 4$.

Since in a circuit it is not necessary to pass through each vertex only once (and considering that $n$ is greater than three by hypothesis), if $n$ is even, it is sufficient to take the minimal covering trail $T_4 := (3,3)$-$(0,3)$-$(0,0)$-$(4,0)$-$(1,3)$-$(1,0)$-$(3,2)$ shown on the left side of Figure \ref{fig:CP_Figure_4}, whereas we will take the covering trail $T_5 := (4,4)$-$(4,0)$-$(0,0)$-$(0,8)$-$(4,0)$-$(-1,5)$-$(7,1)$-$(0,1)$-$(3,4)$ displayed on the right side of Figure \ref{fig:CP_Figure_4} for any odd value of $n$ above four; then we will apply the standard square spiral pattern, matching the extension of the blue edge in $(n+1,n)$ or $(n,n+1)$ (depending on whether $n$ is even or odd, respectively), as shown in Figure \ref{fig:CP_Figure_5}. This guarantees that, for each $n$ above three, all the resulting covering circuits are minimal (i.e., given $n \in \mathbb{N}-\{0, 1, 2, 3\}$ and $F_n$ as described above, it follows that $h(F_n) = 2 \cdot (n-1)$).

\begin{figure}[H]
\begin{center}
\includegraphics[width=\linewidth]{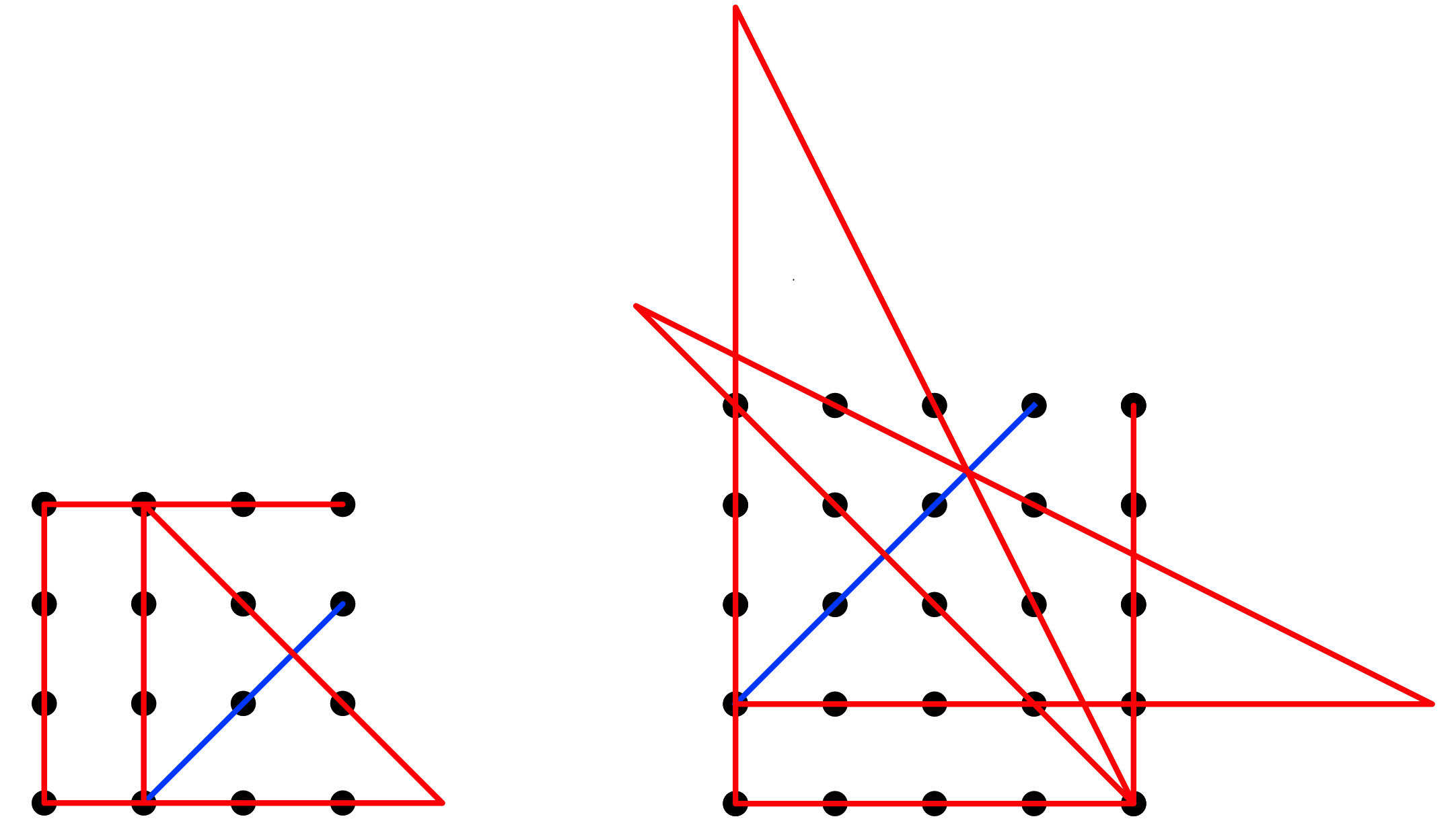}
\end{center}
\caption{The optimal covering trails $T_4$ and $T_5$ (on the left and the right side, respectively) which easily allow us to prove the existence of covering circuits of link length $2 \cdot (n-1)$ for any $G_{n}^2$ such that $n \geq 4$.}
\label{fig:CP_Figure_4}
\end{figure}

\begin{figure}[H]
\begin{center}
\includegraphics[width=\linewidth]{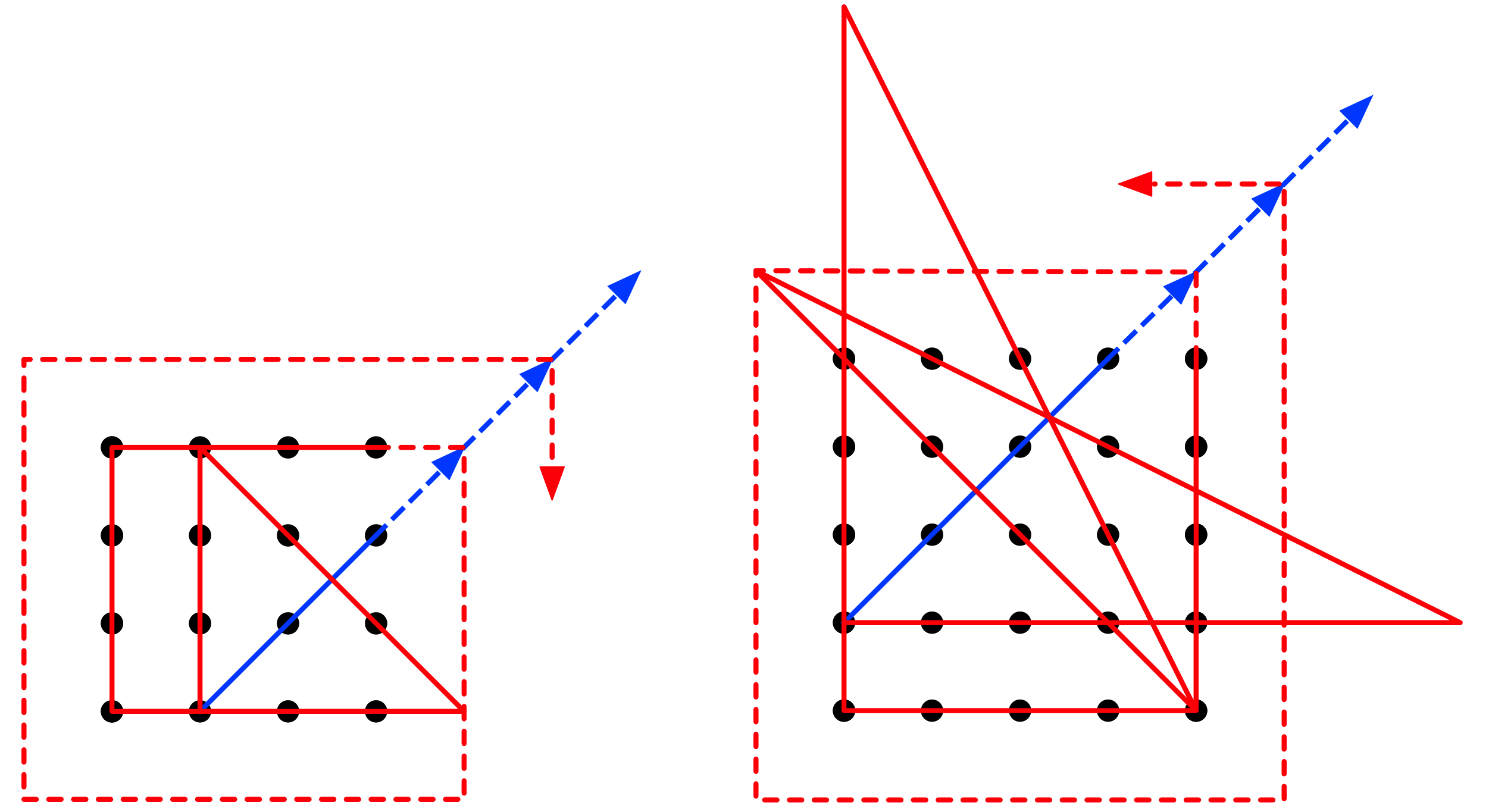}
\end{center}
\caption{A constructive proof of Theorem \ref{Theorem 2.2}, assuming $n > 3$. We take the two covering trails by Figure \ref{fig:CP_Figure_4} and close optimal circuits by extending the blue edge to the top right, while we apply the standard square spiral pattern to the other endpoint (in red).}
\label{fig:CP_Figure_5}
\end{figure}

Therefore, $\min(h\{F_n \}) = \min(h\{T_n \})$ holds for any $n \in \mathbb{N}-\{0,1,3\}$ and the proof of Theorem \ref{Theorem 2.2} is completed.
\end{proof}

\begin{corollary} \label{Corollary 2.1}
If $n \in \{1,3\}$, then $\nexists F_n \in \{F_n \} : h(F_n) = \min(h\{T_n \})$.
\end{corollary}

\begin{proof}
For $n = 1$ the proof is trivial (i.e., any non-degenerate line segment has two distinct endpoints), while Reference \cite{Keszegh:4} covers the case $n = 3$, since $h(T_3 ) = 4$ if and only if $T_3$ (possibly turned clockwise by $90^{\circ}$ or $180^{\circ}$ or $270^{\circ}$) is contained in the graph depicted in Figure \ref{fig:CP_Figure_6} so that if the starting point of $T_3$ belongs to the (blue) half-line $r$, then the endpoint must lay on the green half-line $s$ or on the other half-line $t$, and vice versa.

\begin{figure}[H]
\begin{center}
\includegraphics[width=\linewidth]{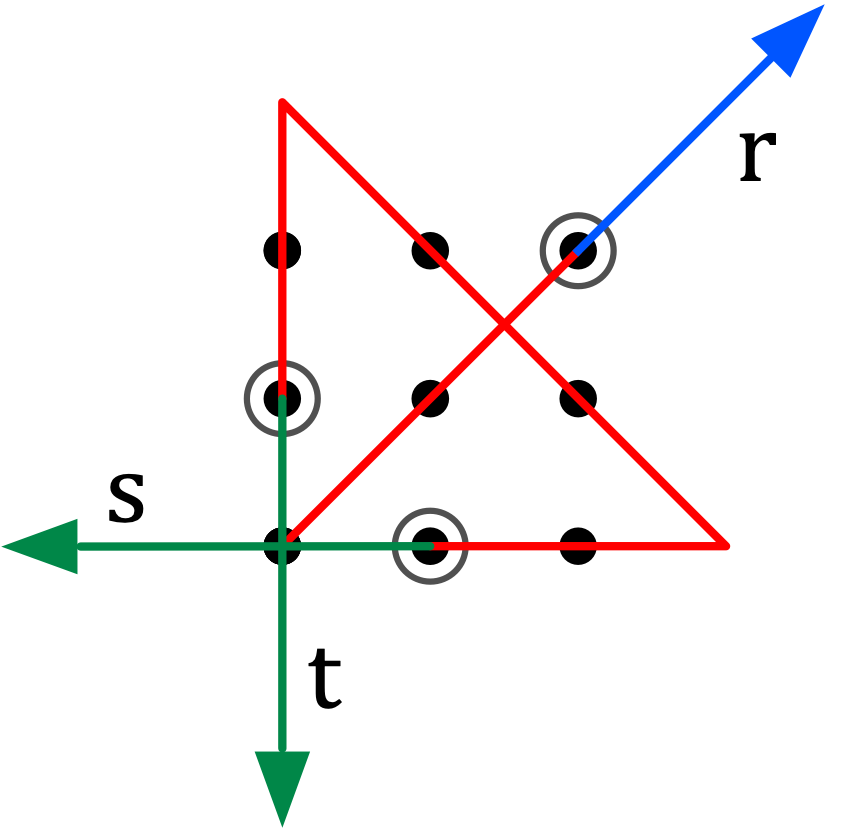}
\end{center}
\caption{The blue line $r$ cannot intersect any of the two green lines. This shows that $\min(h\{F_3 \}) > \min(h\{T_3 \})$ (i.e., $\min(h\{F_3 \}) = \min(h\{T_3 \})+1=5$ since $r$ and $s$ are not collinear).}
\label{fig:CP_Figure_6}
\end{figure}

Thus, $r \cap (s \cup t) = \varnothing$.

Therefore, $n : n \in \{1,3\}$ implies that $\nexists F_n \in \{F_n \} : h(F_n) = \min(h\{T_n \})$.
\end{proof}

\noindent \textbf{Kato's conjecture.} \textit{Let $m \in \mathbb{Z}^+$. $\nexists n : n = 2 \cdot m-1 \wedge \min(h\{C_n \})=\min(h\{T_n \})$.}

\vspace{2mm}

\sloppy This conjecture belongs to the subsection ``Connect the Dots - Lines" of Reference \cite{Pegg:7}. It was first formulated by Toshi Kato (as stated by the author, Edward Taylor Pegg Jr., in the mentioned webpage), and our research has not found any counterexample (e.g., the two self-crossing covering paths for $G_{5}^2$ shown in Figure \ref{fig:CP_Figure_7} turn into a pair of covering circuits only in some non-Euclidean geometries, such as spherical and elliptic geometry, or even by taking into account the non-metrical projective geometry).

Even if we are strongly persuaded that Kato's conjecture is true for any $G_n^2 : n=2 \cdot m-1$, $m \in \mathbb{Z}^+$, a strict proof is still needed.

\begin{figure}[H]
\begin{center}
\includegraphics[width=\linewidth]{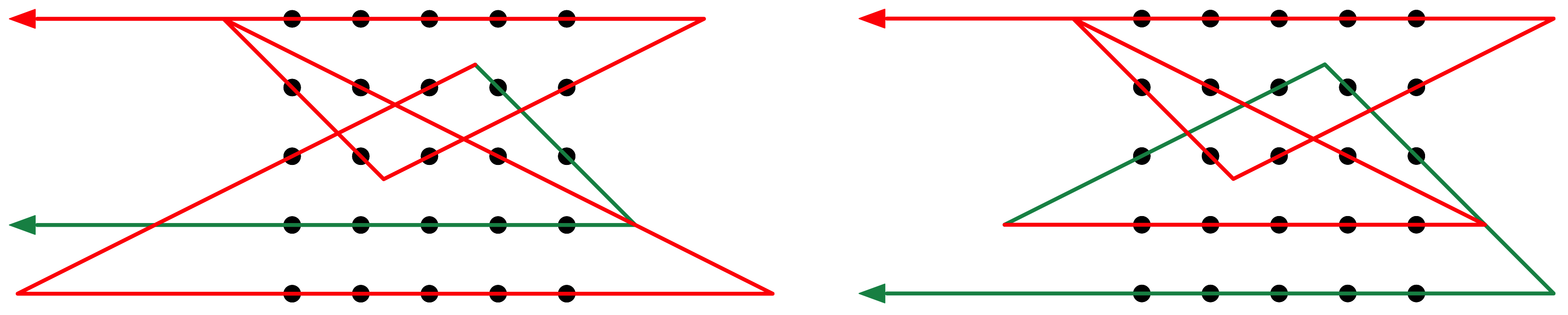}
\end{center}
\caption{A pair of minimum-link covering paths for $G_5^2$ whose first and last edges are parallel each other.}
\label{fig:CP_Figure_7}
\end{figure}

\begin{lemma} \label{Lemma 2.2}
Let $m \in \mathbb{Z}^+$ and let $\tilde{n} := \tilde{n}(m)$ satisfy $\tilde{n}(m) = 2 \cdot \tilde{m}+3$. Then, $\exists F_{\tilde{n}} : h(F_{\tilde{n}}) = \min(h\{ T_{\tilde{n}} \}) = 2 \cdot (\tilde{n}-1)$ and such that the only node of $G_{\tilde{n}}^2$ visited more than once by $F_{\tilde{n}}$ coincides with the beginning and ending vertex of the circuit itself.
\end{lemma}

\begin{proof}
Let $m=1$ so that $\tilde{n}=5$. In order to end the proof, it is sufficient to observe that $F_5 := (4,0)$-$(-1,0)$-$\left(\frac{11}{3},\frac{7}{3} \right)$-$\left(\frac{1}{2},\frac{11}{2} \right)$-$(4,-5)$-$(4,5)$-$(0,1)$-$(0,4)$-$(4,0)$ is a minimum-link covering circuit for $\{0,1,2,3,4\} \times \{0,1,2,3,4\} \subset \mathbb{R}^2$ that visits only the vertex $(4,0)$ more than once.	
\end{proof}

\begin{theorem} \label{Theorem 2.3}
Let $m \in \mathbb{Z}^+$ and let $\tilde{n} := \tilde{n}(m)$ satisfy $\tilde{n}(m) = 2 \cdot \tilde{m}+3$. Then, $\exists P_{\tilde{n}} : h(P_{\tilde{n}}) = 2 \cdot (\tilde{n}-1)$ and such that the Euclidean distance between the starting point and the endpoint is infinitesimal.
\end{theorem}

\begin{proof}
Let us call $\overrightarrow{P_5}$ the covering path for $G_5^2$ shown in Figure \ref{fig:CP_Figure_8}.

\begin{figure}[ht]
\begin{center}
\includegraphics[width=\linewidth]{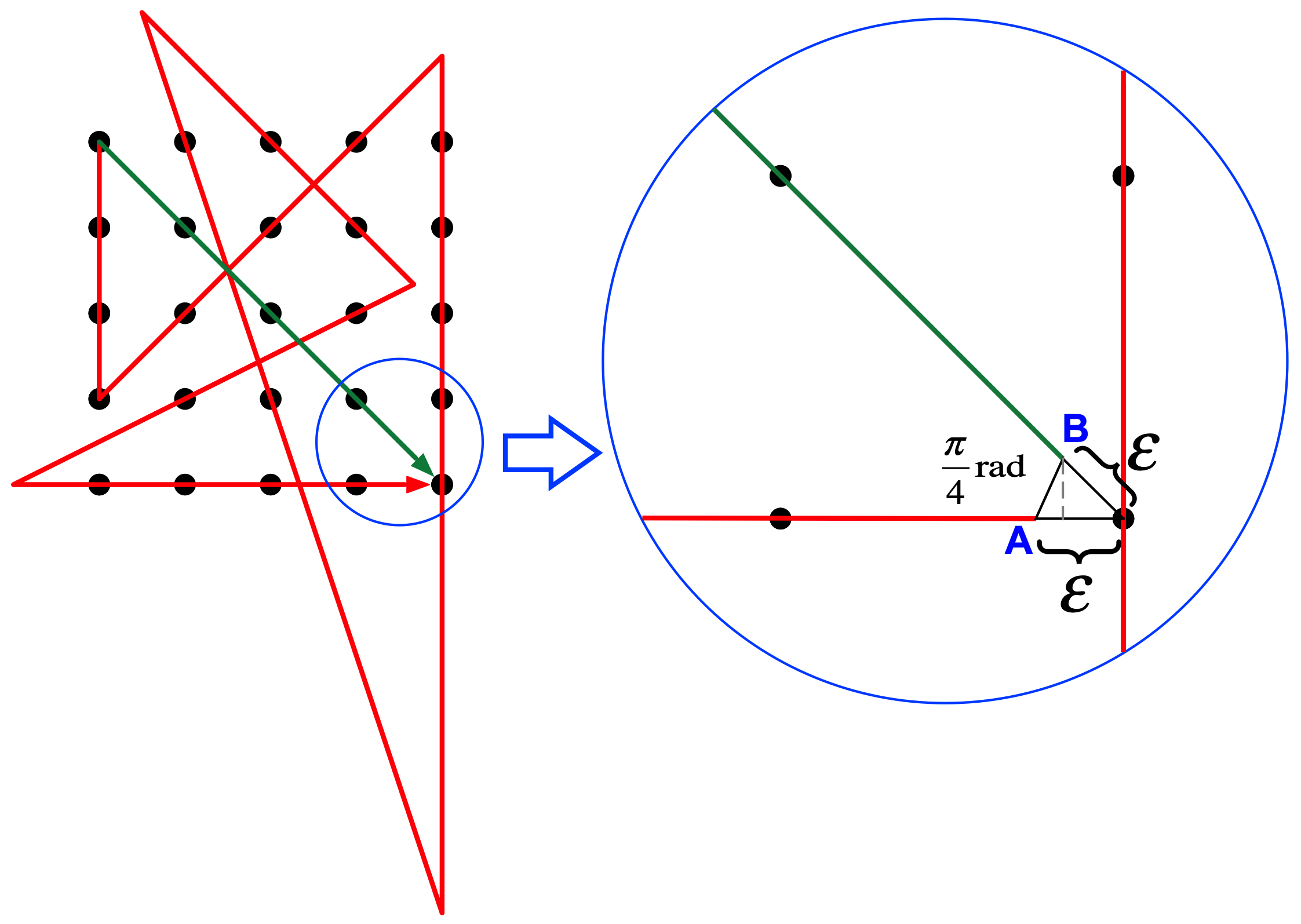}
\end{center}
\caption{The distance between the starting point and the endpoint of the covering path $\protect\overrightarrow{P_5}\!$ can be arbitrarily reduced, and we note that $ h(\protect\overrightarrow{P_5}\!)=\min(h\{T_5 \}) = 8$.}
\label{fig:CP_Figure_8}
\end{figure}

Then, let $\epsilon \in \mathbb{R}^+$. From Figure \ref{fig:CP_Figure_8}, we observe that $\textnormal{A} \equiv (4-\epsilon,0)$, $\textnormal{B} \equiv \left(\frac{1}{\sqrt{2}} \cdot \epsilon,\frac{1}{\sqrt{2}} \cdot \epsilon \right)$, and

\begin{equation} \label{eq:1}
\begin{gathered}
\overline{\rm \textnormal{AB}}=\sqrt{\left( 1-\cos \left(\frac{\pi}{4} \right) \right)^2 \cdot \epsilon^2+\sin^2 \left(\frac{\pi}{4} \right) \cdot \epsilon^2 } = \epsilon \cdot \sqrt{1+\cos^2 \left(\frac{\pi}{4} \right) -2 \cdot \cos\left(\frac{\pi}{4} \right)+\sin^2 \left(\frac{\pi}{4} \right)} \\
\Rightarrow \overline{\rm \textnormal{AB}} = \epsilon \cdot \sqrt{2 \cdot \left(1-\cos \left(\frac{\pi}{4} \right) \right)}. \end{gathered}
\end{equation}

Hence,
\begin{equation} \label{eq:2}
\overline{\rm \textnormal{AB}}=\epsilon \cdot \sqrt{2 \cdot \left(1-\frac{1}{\sqrt{2} }\right)}.
\end{equation}

Now,
\begin{equation} \label{eq:3}
\lim_{\epsilon \to 0} \left( \epsilon \cdot \sqrt{2 \cdot \left(1-\frac{1}{\sqrt{2} }\right)} \right) =\sqrt{2 \cdot \left(1-\frac{1}{\sqrt{2}}\right)} \cdot \lim_{\epsilon \to 0} \epsilon = 0, 
\end{equation}
but $\overrightarrow{P_5}$ is not a cycle, since $\overline{\rm \textnormal{AB}} \nsubseteq \overrightarrow{P_5}$.

On the other hand, if we extend the first and the last edges of $\overrightarrow{P_5}$ we can close the covering circuit $F_5$ mentioned in the proof of Lemma \ref{Lemma 2.2}, and we have already underlined that $F_5$ visits only once each node of $G_5^2$ with the unique exception of the vertex $(4,0)$ itself.

Therefore, by Equations (\ref{eq:2}) and (\ref{eq:3}), the (self-crossing) covering path $\overrightarrow{P_5} := (4- \epsilon,0)$-$(-1,0)$-

\noindent $\left(\frac{11}{3},\frac{7}{3} \right)$-$\left(\frac{1}{2},\frac{11}{2} \right)$-$(4,-5)$-$(4,5)$-$(0,1)$-$(0,4)$-$\left(\frac{1}{\sqrt{2}} \cdot \epsilon,\frac{1}{\sqrt{2}} \cdot \epsilon \right)$ satisfies Theorem \ref{Theorem 2.3}.
\end{proof}

According to Definition \ref{def1.3}, and considering the planar grid $G_n^2$, let us indicate as $\{\{T_n \}-\{P_n \}\}$ the set of all the covering trails that are not paths. Corollary \ref{Corollary 2.2} follows.

\begin{corollary} \label{Corollary 2.2}
Let $n \in \mathbb{Z}^+$ be given and let $\tilde{T_n}$ be a covering trail for $G_n^2$. Then, $\exists \tilde{T_n} \in \{\{T_n \}-\{P_n \}\} : h(\tilde{T_n}) = \min(h\{P_n \})$ if and only if $n \neq 1$.
\end{corollary}

\begin{proof}
Let $n = 1$. We have one point that can be visited using a unique line; trivially, \linebreak $\min(h\{P_1 \})=1$. On the flip side, since $\{P_n \} \cap \{\{T_n \}-\{P_n \}\} = \varnothing $ by hypothesis, it is immediate to conclude that $\min(h\{\{T_n \}-\{P_n \}\}) > 1$ (i.e., a line cannot cross itself). Thus, $n=1 \Rightarrow \min(h\{\{T_n \}-\{P_n \}\}) \neq \min(h\{P_n \})$.

In order to show that $\min(h\{\{T_n \}-\{P_n \}\}) \neq \min(h\{P_n \}) \Rightarrow n=1$, let us invoke Lemma \ref{Lemma 2.1} and consider the proof of Theorem \ref{Theorem 2.2}; it follows that $\forall n \in \mathbb{Z}^+, \exists P_n \in \{P_n \} : h(P_n ) = \min(h\{T_n \})$ (by Lemma \ref{Lemma 2.1}) and also $\min(h\{T_n \}) = \min(h\{\{T_n \}-\{P_n \}\})$ cannot be false if $n$ is above three (as clearly shown in Figure \ref{fig:CP_Figure_5}). Then, the proof of Corollary \ref{Corollary 2.2} only needs to include the pair of minimal covering trails, for $G_2^2$ and $G_3^2$, represented in Figure \ref{fig:CP_Figure_9}, since both such trails visit twice the node $(0,0)$.

\begin{figure}[H]
\begin{center}
\includegraphics[width=\linewidth]{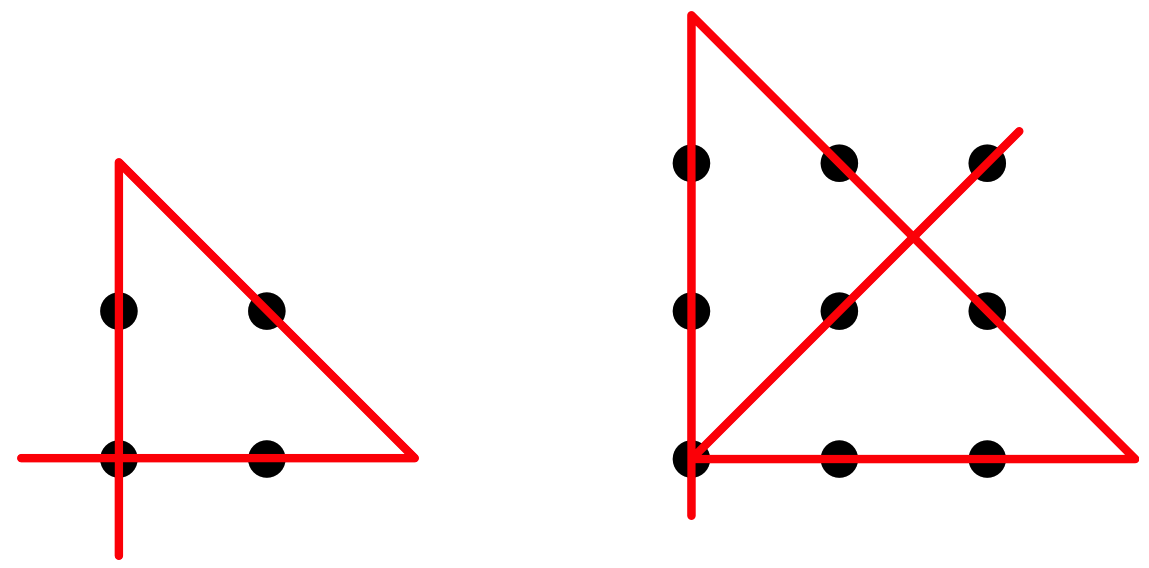}
\end{center}
\caption{Minimum-link covering trails for $G_2^2$ and $G_3^2$; both of them visit twice the node $(0,0)$.}
\label{fig:CP_Figure_9}
\end{figure}

Therefore, $\min(h\{T_n \}) \neq \min(h\{P_n \}) \Leftrightarrow n = 1$, and this concludes the proof of Corollary \ref{Corollary 2.2}.		
\end{proof}


\section{The minimum distance problem} \label{sec:3}

The goal of the present section is to provide a nontrivial upper bound for a variant of the Traveling Salesman Problem (TSP), finding the shortest route to visit every node of $G_n^2$, under the constraint of using only covering trails with $\min(h\{T_n \})$ edges \cite{Wagner:9}.

Thus, for any given $ n \in \mathbb{Z}^+$, we aim to join the $n^2$ vertices of the planar grid $\{0,1, \dots ,n-1\} \times \{0,1, \dots ,n-1\}$ by performing a total of $2 \cdot n-3$ turns, and then we will minimize the total (Euclidean) distance travelled in any $T_n \in \{T_n \} : h(T_n ) = \min(h\{T_n \})$.

In this regard, Figure \ref{fig:CP_Figure_10} shows how to achieve the upper bound given by (\ref{eq:4}); according to the nomenclature introduced by Definition \ref{def1.6}, for any $n \in \mathbb{N}-\{0,1\}$, we have that

\begin{equation} \label{eq:4}
l_u \{T_n \} = \begin{cases}
      3  \hspace{29mm} \mathrm{if}\ \quad n=2 \\
      5 \cdot (1+\sqrt{2}) \hspace{10mm} \mathrm{if}\ \quad n=3 \\
      20+6 \cdot \sqrt{2} \hspace{11mm}  \mathrm{if}\ \quad n=5 \\
      n^2-3+5 \cdot \sqrt{2} \quad \text{otherwise} \\
    \end{cases}\,.
\end{equation}

\begin{figure}[H]
\begin{center}
\includegraphics[width=\linewidth]{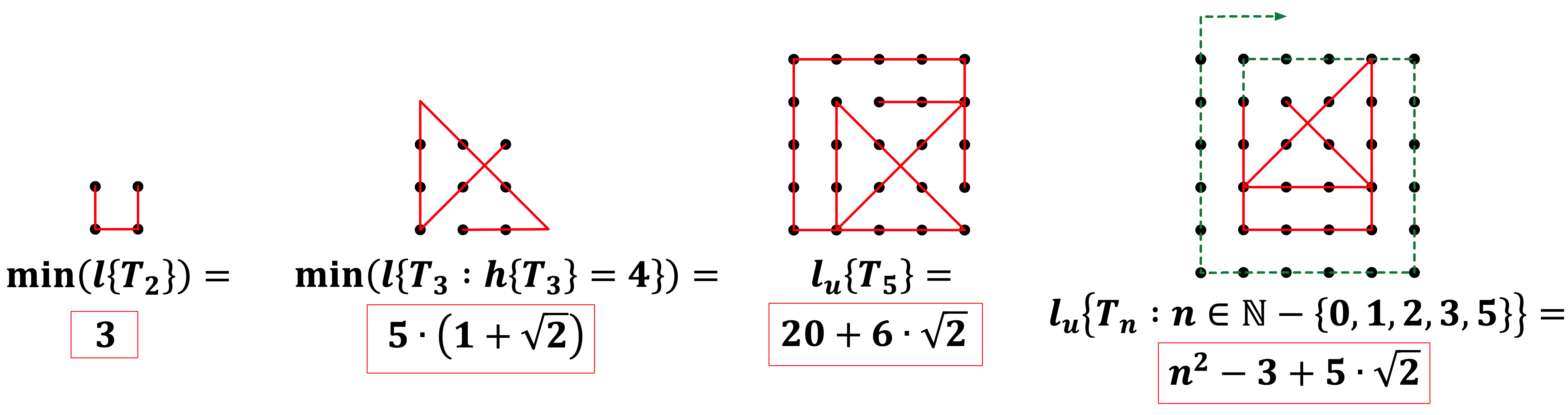}
\end{center}
\caption{Graphical proof that the difference between the upper and the lower bound of the minimum total distance travelled, along any optimal covering trail $T_n : h(T_n ) = \min(h\{T_n \})$, $n \in \mathbb{Z}^+$, cannot be higher than $5 \cdot \sqrt{2}-2$ units.}
\label{fig:CP_Figure_10}
\end{figure}

Since the minimum total distance travelled in any minimal covering trail for $G_n^2$ cannot exceed the upper bound provided by (\ref{eq:4}), and considering that a trivial lower bound is $l_l \{T_n \}=n^2-1$ (see \cite{Keszegh:4}, Fig. 1(a)), we can conclude that if $n$ is above one, then

\begin{equation} \label{eq:5}
n^2-1 \leq \min(l\{T_n : h\{T_n \}=\min(h\{T_n \})\}) \leq n^2-3+5 \cdot \sqrt{2}.
\end{equation}

In details, from Fig. \ref{fig:CP_Figure_10}, we have that $l_u \{T_2 \}=\min(l\{T_2 \})=l((0,1)$-$(0,0)$-$(1,0)$-$(1,1))=3$ units, $l_u \{T_3 \}=\min(l\{T_2 \})=l((2,2)$-$(0,0)$-$(0,3)$-$(3,0)$-$(1,0))=5 \cdot (1+\sqrt{2})$ units, $l_u \{T_5 \}=l((2,3)$-$(4,3)$-$(1,0)$-$(1,3)$-$(4,0)$-$(0,0)$-$(0,4)$-$(4,4)$-$(4,1))=20+6 \cdot \sqrt{2}$ units, and finally \linebreak 
$l_u \{T_{(n=4,6,7,8,\dots)} \}=n^2-3+5 \cdot \sqrt{2}$ units from the standard square spiral pattern applied to the covering trail $T_4=(1,3)$-$(3,1)$-$(0,1)$-$(3,4)$-$(3,0)$-$(0,0)$-$(0,3+c)$, where $c=0$ if and only if $n=4$ and $c=1$ if and only if $n \geq 6$.

\vspace{4mm}

Now, from Reference \cite{Keszegh:4}, Claim 1, we know that the minimum number of edges of any rectilinear covering (or spanning) tree of $G_n^2 : n>2$ (see \cite{Ripa:8}, page 94, Definition 2) is equal to $2 \cdot n-1$ (while we have that $n : n>2 \Rightarrow \min(h\{P_n \})=2 \cdot n-2$ by Lemma 1). Then, if $n$ is above two, it follows that (\ref{eq:5}) can be slightly improved by (\ref{eq:6})

\begin{equation} \label{eq:6}
n^2-1 < \min(l\{T_n : h\{T_n \}=\min(h\{T_n \})\}) \leq n^2-3+5 \cdot \sqrt{2}.
\end{equation}

Thus, for any given $n \in \mathbb{N}-\{0,1,2\}$, the generally proved upper bound of $n^2-3+5 \cdot \sqrt{2}$ units exceeds the minimum total distance travelled by less than $5 \cdot \sqrt{2}-2$ units (e.g., if $n$ is equal to four, then $l_u \{T_4 \}-l_l \{T_4 \}<6 \cdot \sqrt{2}-4$ units).

Now a couple of spontaneous questions arises: \textit{Is $l_u \{T_n \}$ optimal as it is given by (\ref{eq:4}) (so that $l_u \{T_n \}-\min(l\{T_n : h(T_n )=\min(h\{T_n \})\})=0$ for any $n \in \mathbb{N}-\{0,1\}$)?} If not, \textit{how can the maximum gap of $5 \cdot \sqrt{2}-2$ units be significantly reduced for any $n$ above four?}


\section{Conclusion} \label{sec:4}
In this research paper, we have shown that any $G_n^2$ admits a covering path with the same link length of a minimum-link covering trail for the given planar grid. In particular, if $n$ is even, we can always find a covering cycle characterized by a link length as above.

Although Theorem \ref{Theorem 2.2} proves that a covering circuit having the same number of lines of the minimum-link covering trail for $G_n^2$ exists if and only if $n \in \mathbb{N}-\{0,1,3\}$, an important open problem is still waiting for a strict answer: \textit{Is it possible to prove or disprove Kato's conjecture which implies, for any $G_n^2$ such that $n$ is odd, the non-existence of a covering cycle of link length $2 \cdot (n-1)$?}

Lastly, for any given $n$ above three, a second open problem is to \textit{determine which is the minimum total (Euclidean) distance travelled to visit each node of $G_n^2$ with a minimum-link covering trail}.


\section*{Acknowledgments}

We sincerely thank the young mathematician Flavio Niccol\`o Baglioni for helping and supporting us in developing the proof of Lemma \ref{Lemma 2.1}, Aldo Tragni who first discovered the pattern by Figures \ref{fig:CP_Figure_1} and \ref{fig:CP_Figure_2}, and Stefania Lugano who cleverly found the (peculiar) upper bound $l_u (T_5 )=20+6 \cdot \sqrt{2}$ units.


\bibliographystyle{plain}
\bibliography{Shortest_polygonal_chains_covering_each_planar_square_grid_arXiv}


\section*{Appendix}

As mentioned at the beginning of Section \ref{sec:2}, this appendix is devoted to showing the proof of Lemma \ref{Lemma 2.1} which was originally provided in 2021.

\begin{proof} For any given positive integer $n$, we constructively prove the existence of a covering path for $G_n^2$ which is characterized by the same link length of the minimum-link covering trail for the same planar grid (i.e., in the present appendix we will show that $\min(h\{P_n \})=\min(h\{T_n \})=2 \cdot (n-1)$ for any $n \in \mathbb{N}-\{0,1,2\}$, while $\min(h\{P_1 \})=1$ and $\min(h\{P_2 \})=3)$.

Since $\min(h\{P_n \})=\min(h\{T_n \})$ holds for any $n$ below four by Figure \ref{fig:CP_Figure_3}, from here on, let us assume that $n : n \geq 4$.

Now, the general algorithm we use to prove Lemma \ref{Lemma 2.1} is as follows:
\vspace{0mm}
\begin{enumerate}
  \item 	Construct the \textit{bottom triangular spiral} $B_n := (0,0)$-$(n-2,0)$-$(0,n-2)$-$(0,1)$-$(n-4,1)$- \linebreak $(1,n-4)$-$(1,2)$-$\dots$ and continue in this way for a total of $n-2$ edges so that the only missed point in the middle of $B_n$ is $\textnormal{P}_1(n)$ (see Figure \ref{fig:CP_Figure_11}).
  \item 	Since $\textnormal{P}_1(4) \equiv (x(\textnormal{P}_1(4)), y(\textnormal{P}_1(4))) \equiv (0,1)$, it follows that $\textnormal{P}_1 (n) \equiv \Bigl(x(\textnormal{P}_1 (4))+\floor*{\frac{n-2}{3}},$ \linebreak $y(\textnormal{P}_1 (4))+\floor*{\frac{n-3}{3}} \Bigr) \equiv \Bigl(\floor*{\frac{n-2}{3}},\floor*{\frac{n}{3}} \Bigr)$ (where $\floor*{q}$ indicates the floor function which takes as input a real number $q$ and returns the greatest integer less than or equal to $q$).
  \item 	Construct the \textit{top triangular spiral} $U_n := (n-1,n-1)$-$(0,n-1)$-$(n-1,0)$-$(n-1,n-2)$-$(2,n-2)$-$(n-2,2)$-$(n-2,n-3)$-$(4,n-3)$-$\dots$ and so on, until a total of $n-1$ edges is reached. Thus, we get a unique, unvisited, point in the middle of $U_n$ and we call it $\textnormal{P}_2 (n)$ (see Figure \ref{fig:CP_Figure_12}).
    \item Since $\textnormal{P}_2 (4) \equiv (x(\textnormal{P}_2 (4)),y(\textnormal{P}_2 (4))) \equiv (2,2)$, we have that $\textnormal{P}_2 (n) \equiv \Bigl(x(\textnormal{P}_2 (4))-\floor*{\frac{n-4}{3}},$ \linebreak $y(\textnormal{P}_2 (4))-\floor*{\frac{n-2}{3}} \Bigr) \equiv \Bigl(n-2-\floor*{\frac{n-4}{3}},n-2-\floor*{\frac{2-n}{3}} \Bigr)$.
      \item 	Draw the line $r$ that passes through the two points $\textnormal{P}_1 (n)$ and $\textnormal{P}_2 (n)$, then extend the first/last segment of $B_n$ and the last/first segment of $U_n$ accordingly to Figure \ref{fig:CP_Figure_13}, so that we connect $B_n$ to $U_n$ through $r : \overline{\rm \textnormal{P}_1 \mathit{(n)} \textnormal{P}_2 \mathit{(n)}} \in r$ (it is trivial to point out that $0<\frac{y(\textnormal{P}_2 (n))-y(\textnormal{P}(n))}{x(\textnormal{P}_2 (n))-x(\textnormal{P}_1 (n))} < 1)$. Moreover, since we are assuming that $n : n \geq 4$, if $n : n \neq 0 \pmod{3}$, then it is always possible to apply the usual square spiral pattern to $B_n$ or $U_n$ to prove that $\min(h\{P_n \})=\min(h\{T_n \}) \Rightarrow  \min(h\{P_{(n+1)} \})=\min(h\{T_{(n+1)}\})$.
        \item We observe that $\overline{\rm \textnormal{P}_1 \mathit{(n)} \textnormal{P}_2 \mathit{(n)}} \in r \Leftrightarrow r : \frac{x-x(\textnormal{P}_1 (n))}{(x(\textnormal{P}_2 (n))-x(\textnormal{P}_1 (n))}=\frac{y-y(\textnormal{P}_1 (n))}{y(\textnormal{P}_2 (n))-y(\textnormal{P}_1 (n))}$. Hence, \linebreak $r : \frac{x-\floor*{\frac{n-2}{3}}}{n-2-\floor*{\frac{n-4}{3}} -\floor*{\frac{n-2}{3}}}=\frac{y-\floor*{\frac{n}{3}}}{n-2-\floor*{\frac{n-2}{3}} -\floor*{\frac{n}{3}}}$.
\end{enumerate}

\begin{figure}[H]
\begin{center}
\includegraphics[width=\linewidth]{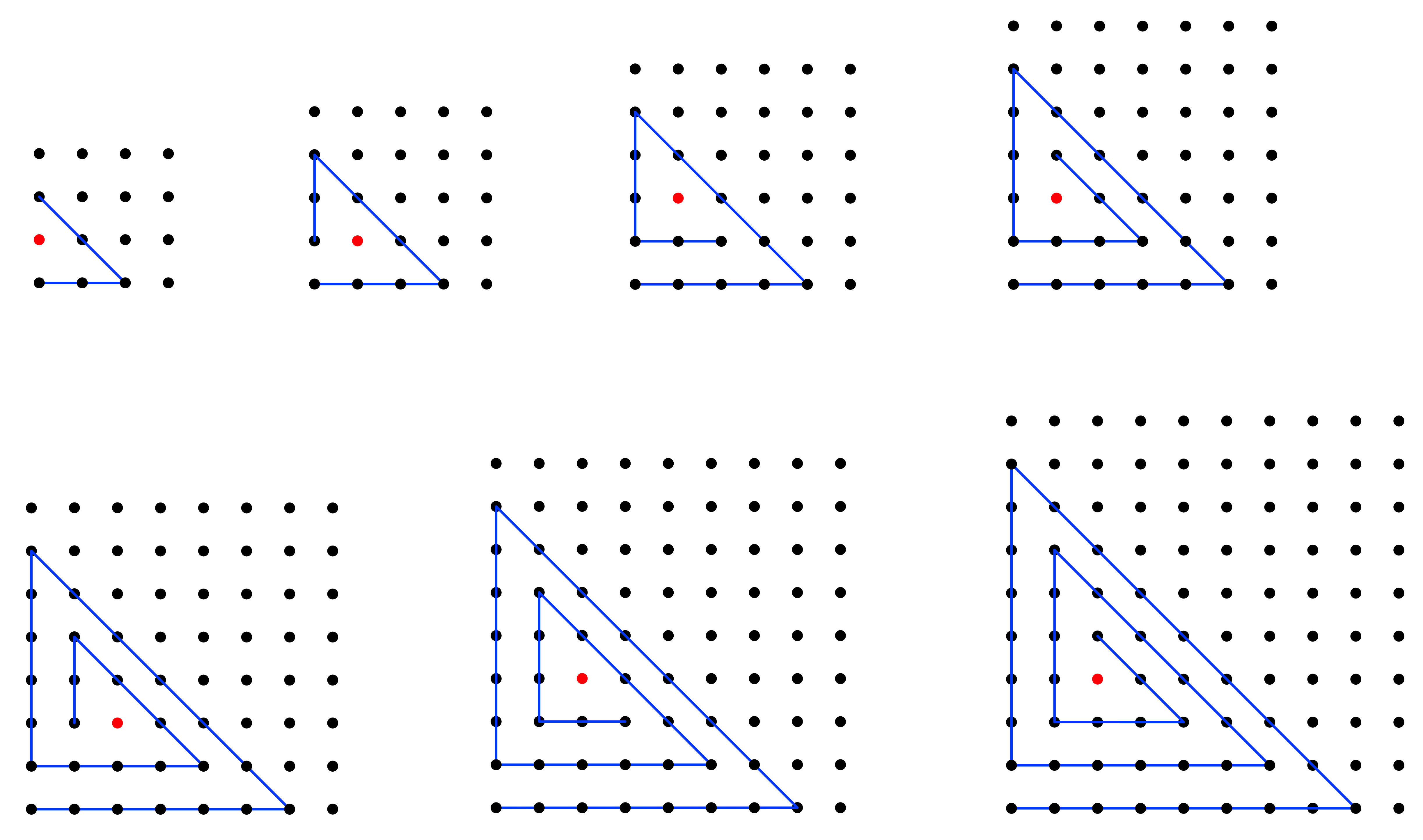}
\end{center}
\caption{The bottom triangular spiral $B_n := (0,0)$-$(n-2,0)$-$(0,n-2)$-$(0,1)$-$(n-4,1)$- \linebreak $(1,n-4)$-$(1,2)$-$\dots$ for $G_{(n=4,5,6,7,8,9,10)}^2$.}
\label{fig:CP_Figure_11}
\end{figure}

\begin{figure}[H]
\begin{center}
\includegraphics[width=\linewidth]{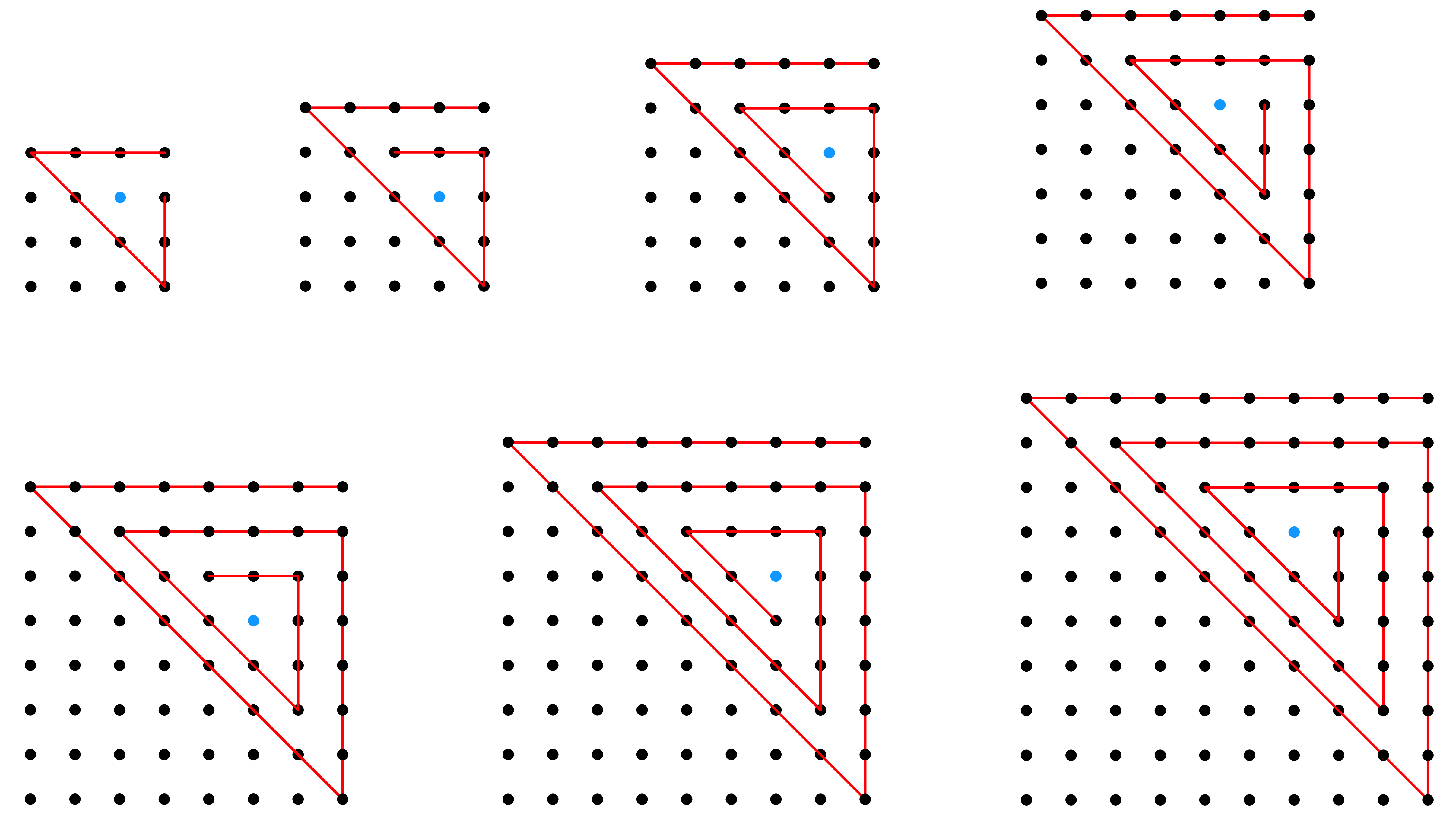}
\end{center}
\caption{The top triangular spiral $U_n := (n-1,n-1)$-$(0,n-1)$-$(n-1,0)$-$(n-1,n-2)$- \linebreak $(2,n-2)$-$(n-2,2)$-$(n-2,n-3)$-$(4,n-3)$-$\dots$ for $G_{(n=4,5,6,7,8,9,10)}^2$.}
\label{fig:CP_Figure_12}
\end{figure}

\begin{figure}[ht]
\begin{center}
\includegraphics[width=\linewidth]{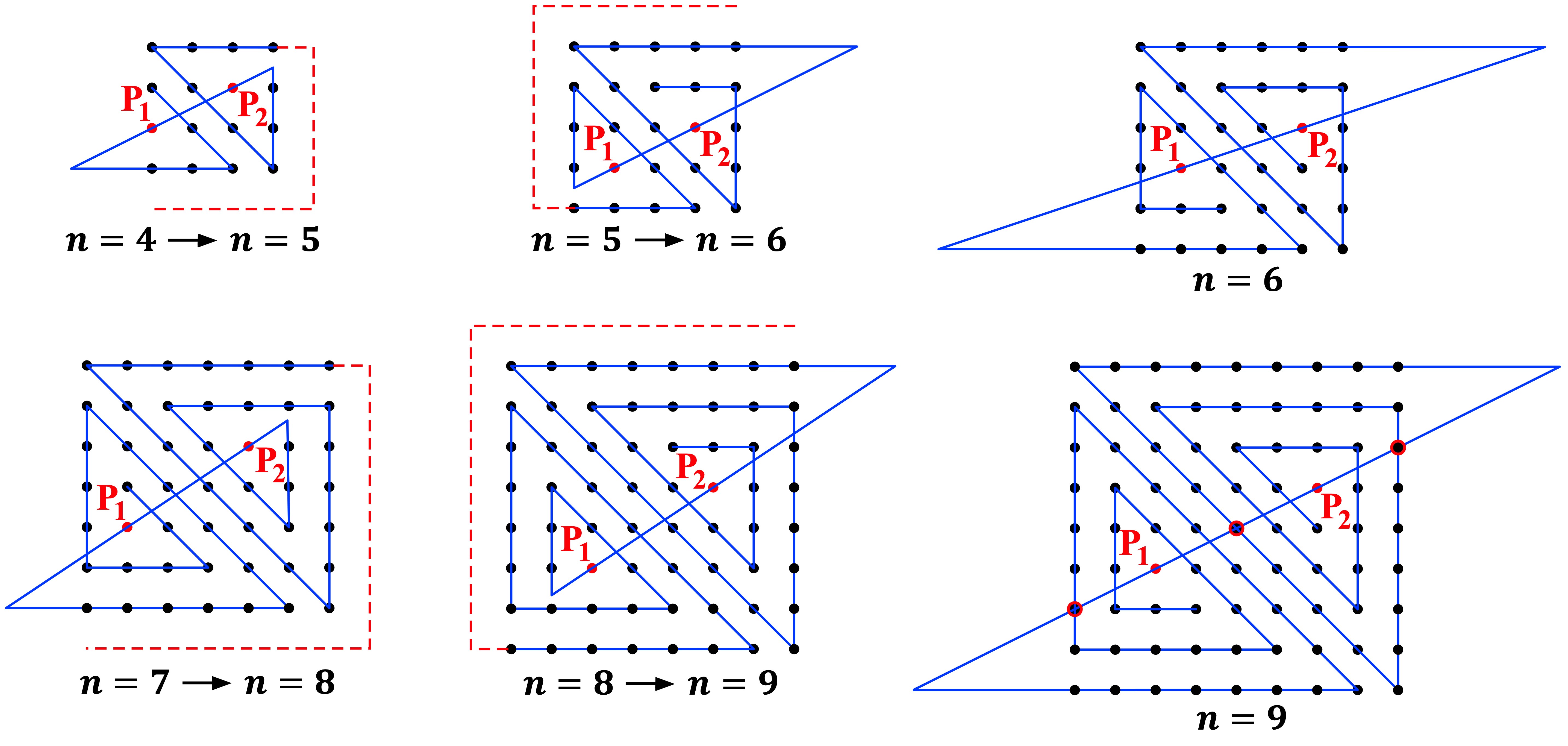}
\end{center}
\caption{The \textit{mixed} spiral pattern which proves that $\min(h\{P_n \})=\min(h\{T_n \}) \Rightarrow$ $\min(h\{P_{(n+1)} \})=\min(h\{T_{(n+1)} \})$ for any $n : n \not{\!\!\equiv} \hspace{1mm} 0 \pmod {3}$. In particular, given $n \in \mathbb{N}-\{0,1,2,3\}$ as usual, there are five distinct collisions between $G_n^2$ and the edge comprising the line segment $\overline{\rm \textnormal{P}_1 \mathit{(n)} \textnormal{P}_2 \mathit{(n)}}$, namely $\textnormal{Q}_{(t=0,1,2,3,4)} := \textnormal{Q}_t (n) \equiv \left(x(\textnormal{P}_1 (n))+\frac{n+3}{6} \cdot (t-1),y(\textnormal{P}_2 (n))+\frac{n-3}{6} \cdot (t-1)\right)$, if and only if $n : (n \equiv 3 \pmod {6}  \wedge n>3)$ (i.e., if we assume that $n : (n \equiv 3 \pmod {6} \wedge n>3)$, then $\textnormal{Q}_1 (n) \equiv \textnormal{P}_1 (n) \wedge \textnormal{Q}_3 (n) \equiv \textnormal{P}_2 (n)$ holds for any $n$ as above).}
\label{fig:CP_Figure_13}
\end{figure}

The mixed square spiral pattern, depicted in Figure \ref{fig:CP_Figure_13}, guarantees that if there are no \linebreak accidental collisions for any $n : n \not{\!\!\equiv} \hspace{1mm} 0 \pmod {3}$, then $\min(h\{P_n \})=\min(h\{T_n \})$ will always hold.

Now let $j \in \mathbb{Z}^+$ be given. We first prove that not a single collision can occur between the line $r : n=3 \cdot j+1$ and any element of the grid which is distinct from $\textnormal{P}_1 \equiv (j-1,j)$ and $\textnormal{P}_2 \equiv (2 \cdot j,2 \cdot j)$ (i.e., we need to show that any node of $G_{(3 \cdot j+1)}^2$ cannot be visited twice by $B_{(3 \cdot j+1)} \cup U_{(3 \cdot j+1)} \cup (r : n = 3 \cdot j+1))$; and then, in order to end this proof, we will verify the same property for any $n : n \equiv 2 \pmod {3}$ by considering the line $r : n=3 \cdot j+2$ (it is not hard to prove that, inter alia, accidental collisions occur if and only if $n : n \equiv 3 \pmod {6}$, see Figure \ref{fig:CP_Figure_13}).

Accordingly, let us begin from the case $n : n=3 \cdot j+1$ and check that the system (\ref{eq:7})

\begin{equation} \label{eq:7}
\begin{cases}
     \frac{x-\floor*{\frac{n-2}{3}}}{n-2-\floor*{\frac{n-4}{3}}-\floor*{\frac{n-2}{3}}}=\frac{y-\floor*{\frac{n}{3}}}{n-2-\floor*{\frac{n-2}{3}}-\floor*{\frac{n}{3}}} \\
      n \equiv 1 \pmod{3} \\
      n>1 \\
    \end{cases}\
\end{equation}

admits no integer solution in the domain $\mathcal{H}(n) := \{0,1,\dots,n-1 \}-\{x(\textnormal{P}_1 (n)),x(\textnormal{P}_2 (n))\}$, for $x(\textnormal{P}_1 (n)) = \floor*{\frac{n-2}{3}}$ and 
$x(\textnormal{P}_2 (n)) = n-2-\floor*{\frac{n-4}{3}}$.

We can rewrite (\ref{eq:7}) as

\begin{equation} \label{eq:8}
\begin{cases}
     \frac{x-j+1}{2 \cdot j-j+1}=\frac{y-j}{j} \\
     j+1 \neq 0 \\
     j \neq 0 \\
    \end{cases}\
\end{equation}

Consequently, (\ref{eq:8}) provides a necessary condition for proving Lemma \ref{Lemma 2.1}; we need to verify that $x_1=j-1$ and $x_2=2 \cdot j$ map the domain $\{0,1, \dots ,3 \cdot j \}$ into the only pair of integer values of the dependent variable, $y := y(x)$, such that $(x,y) \in G_{(3 \cdot j+1)}^2$.

Since $j$ is a positive integer by definition, we observe that the conditions $j \neq -1$ and $j \neq 0$, from (\ref{eq:8}), are always satisfied.

Hence,

\begin{equation} \label{eq:9}
 y=\frac{x \cdot j+2 \cdot j}{j+1},	
\end{equation}

so that $y \in \mathbb{N}_0 \Rightarrow (j+1) | j \cdot (x+2)$.

Here, $\textnormal{gcm}(j+1,j)=1$ and it follows that $y \in \mathbb{N}_0 \Rightarrow (j+1) | (x+2)$.

Thus, a necessary condition for $y$ to be a nonnegative integer is given by $x_k=(j+1) \cdot k-2$, $k \in \mathbb{N}_0$, and a collision between $B_{(3 \cdot j+1)} \cup U_{(3 \cdot j+1)} \cup (r : n=3 \cdot j+1)$ and $G_{(3 \cdot j+1)}^2$ can occur if and only if $x_k \in \{0,1,\dots,n-1\}$, so that it is immediate to verify that $k \in \{1,2\}$ and $x_{1,2}=\{j-1,2 \cdot j\}$ (i.e., if $k$ is above two, then $x_k \geq 3 \cdot j+1 \Rightarrow x_{(k \geq 3)} \geq n$, whereas $k=0$ returns a negative value of $x_k$ and this automatically implies that $\nexists k \neq (1 \lor 2) : x_k \in \{0,1,\dots,n-1 \}$).

From the previous relation, we conclude that $\nexists x \in \mathcal{H}(3 \cdot j+1) : y \in \mathbb{N}_0$ and this proves that our algorithm returns a (self-crossing) covering path of link length $2 \cdot n-2$ for any given $G_{(3 \cdot j+1)}^2 : j \in \mathbb{Z}^+$.

We now repeat the process for $G_n^2 : (n \equiv 2 \pmod {3} \wedge n>3$), by assuming that $n : n=3 \cdot j+2$. In this case, from $\frac{x-\floor*{\frac{n-2}{3}}}{n-2-\floor*{\frac{n-4}{3}}-\floor*{\frac{n-2}{3}}}=\frac{y-\floor*{\frac{n}{3}}}{n-2-\floor*{\frac{n-2}{3}}-\floor*{\frac{n}{3}}}$ (see (\ref{eq:7})), we get

\begin{equation} \label{eq:10}
\begin{cases}
     \frac{x-j}{j+1}=\frac{y-j}{j} \\
     j \neq -1 \hspace{8mm}. \\
     j \neq 0 \\
    \end{cases}\	
\end{equation}

The conditions $j \neq -1$ and $j \neq 0$ from (\ref{eq:10}) are redundant ($j \geq 1$ by definition), so it follows that

\begin{equation} \label{eq:11}
y=\frac{x \cdot j+j}{j+1},
\end{equation}

and then $y \in \mathbb{N}_0 \Rightarrow (j+1) | j \cdot (x+1)$.

Since $\textnormal{gcm}(j+1,j)=1$, $y \in \mathbb{N}_0 \Rightarrow (j+1) | (x+1)$.

Let $k \in \mathbb{N}_0$ , as usual. If we select $k=0$, then we have $x_{(k=0)}=-1 \Rightarrow \nexists k=0 : x_k \in \{0,1,\dots,n-1\} \supset \mathcal{H}(n)$, whereas if we take $k \geq 3$, then $x_{(k \geq 3)} \geq (3 \cdot j+1)-1 \Rightarrow x_{(k \geq 3)} \geq n \Rightarrow \nexists k \geq 3 : x_k \in \{0,1,\dots,n-1\} \supset \mathcal{H}(n)$. And lastly, $k \in \{1,2\}$ implies that $x_k \notin \mathcal{H}(3 \cdot j+2)$ by construction (i.e., $x_1=j=\frac{3 \cdot j+2-2}{3}=\frac{n-2}{3}$ and $x_2=2 \cdot j+1=\frac{6 \cdot j+4-1}{3}=\frac{2 \cdot n-1}{3}$ obviously match the first coordinates of $\textnormal{P}_1 (n) \equiv (x(\textnormal{P}_1 (4))+\floor*{\frac{n-2}{3}},y(\textnormal{P}_1 (4))-1+\floor*{\frac{n}{3}})$ and $\textnormal{P}_2 (n) \equiv (x(\textnormal{P}_2 (4))-\floor*{\frac{n-4}{3}},y(\textnormal{P}_2 (4))-\floor*{\frac{n-2}{3}})$, respectively).

Since we have shown that $n : (n \equiv 2 \pmod {3} \wedge n>3)$ is a sufficient condition for applying the mixed square spiral pattern to the path returned by the fifth step of our general algorithm, getting a covering path with $2 \cdot (n+1)-2$ edges also for any $G_{(n+1)}^2 : (n+1) \equiv 0 \pmod {3}$, bearing in mind the previous result (which proves the existence of minimal covering paths for any $n : (n \equiv 1 \pmod {3} \wedge n>3)$) and considering that Figure \ref{fig:CP_Figure_3} provides (well-known) covering paths for each $n$ below five, we can finally conclude that, $\forall n \in \mathbb{Z}^+$, $\nexists x \in \mathcal{H}(n) : y \in \mathbb{N}_0$.

Therefore, for any given $n \in \mathbb{Z}^+$, there exists a covering path for $G_n^2$ which is characterized by a link length of $\min(h\{T_n \})$, and this concludes the proof of Lemma \ref{Lemma 2.1}.
\end{proof}

\end{document}